\documentclass[12pt,reqno]{amsart}

\usepackage{amssymb}
\usepackage{verbatim}

\newcommand{\kap}{\kappa}
\newcommand{\sig}{\sigma}

\newcommand{\seq}{\subseteq}

\newcommand{\longc}{,\ldots,}
\newcommand{\longe}{=\dotsb=}
\newcommand{\longp}{+\dotsb+}
\newcommand{\longle}{\le\dotsb\le}
\newcommand{\longge}{\ge\dotsb\ge}
\newcommand{\longu}{\cup\dotsb\cup}

\renewcommand{\ge}{\geqslant}
\renewcommand{\le}{\leqslant}

\newcommand{\lpr}{\left(}
\newcommand{\rpr}{\right)}
\newcommand{\lfr}{\left\{}
\newcommand{\rfr}{\right\}}
\newcommand{\lfl}{\left\lfloor}
\newcommand{\rfl}{\right\rfloor}
\newcommand{\lcl}{\left\lceil}
\newcommand{\rcl}{\right\rceil}

\theoremstyle{plain}
\newtheorem{lemma}{Lemma}
\newtheorem{theorem}{Theorem}
\newtheorem{corollary}{Corollary}
\newtheorem{proposition}{Proposition}
\newtheorem{conjecture}{Conjecture}
\newtheorem{primetheorem}{Theorem}

\newcommand{\refl}[1]{~\ref{l:#1}}
\newcommand{\reft}[1]{~\ref{t:#1}}
\newcommand{\refc}[1]{~\ref{c:#1}}
\newcommand{\refj}[1]{~\ref{j:#1}}
\newcommand{\refp}[1]{~\ref{p:#1}}
\newcommand{\refs}[1]{~\ref{s:#1}}
\newcommand{\refb}[1]{~\cite{b:#1}}
\newcommand{\refe}[1]{~\eqref{e:#1}}

\title[Consecutive integers in sumsets]%
  {Consecutive integers \\ in high-multiplicity sumsets}
\author{Vsevolod F. Lev}
\address{Department of Mathematics, The University of Haifa at Oranim,
  Tivon 36006, Israel.}
\email{seva@math.haifa.ac.il}
\thanks{The author gratefully acknowledges the support of the Georgia
Institute of Technology and the Fields Institute, which he was visiting
while conducting his research.}

\begin{document}
\baselineskip=16pt

\begin{abstract}
Sharpening (a particular case of) a result of Szemer\'edi and Vu \refb{sv}
and extending earlier results of S\'ark\"ozy \refb{s} and ourselves
\refb{la}, we find, subject to some technical restrictions, a sharp threshold
for the number of integer sets needed for their sumset to contain a block of
consecutive integers of length, comparable with the lengths of the set
summands.

A corollary of our main result is as follows. Let $k,l\ge 1$ and $n\ge 3$ be
integers, and  suppose that $A_1\longc A_k\seq[0,l]$ are integer sets of size
at least $n$, none of which is contained in an arithmetic progression with
difference greater than $1$. If $k\ge2\lcl(l-1)/(n-2)\rcl$, then the sumset
$A_1\longp A_k$ contains a block of consecutive integers of length $k(n-1)$.
\end{abstract}

\maketitle

\section{Background and summary of results}

The \emph{sumset} of subsets $A_1\longc A_k$ of an additively written
group is defined by
  $$ A_1\longp A_k
               := \{ a_1\longp a_k\colon a_1\in A_1\longc a_k\in A_k \}; $$
if $A_1\longe A_k=A$, this is commonly abbreviated as $kA$. In the present
paper we will be concerned exclusively with the group of integers, in which
case a well-known phenomenon occurs: if all sets $A_i$ are dense, and their
number $k$ is large, then the sumset $A_1\longp A_k$ contains long arithmetic
progressions. There are numerous ways to specialize this statement,
indicating the exact meaning of ``dense'', ``large'', and ``long'', but in
our present context the following result of S\'ark\"ozy is the origin of
things.

\begin{theorem}[\protect{S\'ark\"ozy \cite[Theorem~1]{b:s}}]\label{t:sarkozy}
Let $l\ge n\ge 2$ be integers, and write $\kap:=\lcl(l+1)/(n-1)\rcl$. Then,
for every integer set $A\seq[1,l]$ with $|A|=n$, there exist positive
integers $d\le\kap-1$ and $k<118\kap$ such that the sumset $kA$ contains $l$
consecutive multiples of $d$.
\end{theorem}

In \refb{la} we established a sharp version of this result, replacing the
factor $118$ with $2$ (which is the best possible value, as conjectured by
S\'ark\"ozy) and indeed, going somewhat further.
\begin{theorem}[\protect{Lev \cite[Theorem~1]{b:la}}]\label{t:lev}
Let $n\ge 3$ and $l\ge n-1$ be integers, and write
$\kap:=\lfl(l-1)/(n-2)\rfl$. Then for every integer set $A\seq[0,l]$ with
$0,l\in A,\ \gcd(A)=1$, and $|A|=n$, and every integer $k\ge 2\kap$, we have
  $$ [\kap(2l-2-(\kap+1)(n-2)), kl-\kap(2l-2-(\kap+1)(n-2))] \seq kA. $$
\end{theorem}

The complicated-looking interval appearing in the statement of Theorem
\reft{lev} is best possible and, in general, cannot be extended even by $1$
in either direction. As the interested reader will easily check, this
interval (strictly) includes $[\kap l,(k-\kap)l]$ as a subinterval;
consequently, if $k\ge 2\kap+1$, then its length exceeds $l$. At the same
time, if $k\le 2\kap$, then the sumset $kA$ may fail to contain a block of
consecutive integers of length $l$; see example below. Thus, $2\kap+1$ is the
smallest value of $k$ such that, with $A$ as in Theorem \reft{lev}, the
sumset $kA$ is guaranteed to contain a block of consecutive integers of
length $l$.

At first sight, Theorem \reft{lev} is weaker than Theorem \reft{sarkozy} in
imposing the extra assumptions $0,l\in A$ and $\gcd(A)=1$. It is explained in
\refb{la}, however, that these assumptions are merely of normalization
nature, and a refinement of Theorem \reft{sarkozy}, with the bound $118\kap$
replaced by $2\kap+2$ (and some other improvements), is deduced from Theorem
\reft{lev} in a relatively straightforward way.

We notice that Theorem \reft{lev} yields a sharp result about the function
$f(n,k,l)$, introduced in \refb{sv}. This function is defined for positive
integers $k$ and $l\ge n\ge 2$ to be the largest number $f$, such that for
every $n$-element integer set $A\seq[1,l]$, the sumset $kA$ contains an
arithmetic progression of length $f$. (The \emph{length} of an arithmetic
progression is the number of its terms, less $1$). As indicated in \refb{sv},
``many estimates for $f(n,k,l)$ have been discovered by Bourgain, Freiman,
Halberstam, Green, Ruzsa, and S\'ark\"ozy''. It is worth noting in this
connection that Theorem \reft{lev} establishes the \emph{exact value} of this
function for $k$ large; namely, it is easy to deduce from Theorem \reft{lev}
(and keeping in mind the trivial example $A=[1,n]$) that
  $$ f(n,k,l)=k(n-1);\quad k\ge 2\lfl(l-2)/(n-2)\rfl+2. $$
This, to our knowledge, remains the only situation where the value of
$f(n,k,l)$ is known precisely.

An obvious shortcoming of Theorem \reft{lev} is that it applies only to
identical set summands. Potentially distinct summand are dealt with by
Szemer\'edi and Vu in \refb{sv}. A particular case of their result, to be
compared with Theorems \reft{sarkozy} and \reft{lev}, is as follows.
\newpage
\begin{theorem}[Szemer\'edi-Vu
      \protect{\cite[particular case of Corollary 5.2]{b:sv}}]\label{t:szvu}
There exist positive absolute constants $C$ and $c$ with the following
property. Suppose that $k,l\ge 1$ and $n\ge 2$ are integers, and
 $A_1\longc A_k\seq[0,l]$ are integer sets, having at least $n$ elements each.
If $k>Cl/n$, then the sumset $A_1\longp A_k$ contains an arithmetic
progression of length at least $ckn$.
\end{theorem}

Though the proof of Theorem \reft{szvu}, presented in \refb{sv}, is
constructive, the constants $C$ and $c$ are not computed explicitly. Indeed,
the argument leads to excessively large values of these constants, which may
present a problem in some applications.

The goal of this note is to merge the best of the two worlds, extending
Theorem \reft{lev} onto distinct set summands, or, equivalently, proving a
sharp analogue if Theorem \reft{szvu}, with $Cl/n$ and $ckn$ replaced with
best possible expressions. A result of this sort (up to a technical
restriction, addressed below) is an almost immediate corollary from the
following theorem, proven in Section \refs{proofs}.
\begin{theorem}\label{t:main}
Let $k,l\ge 1$ and $n\ge 3$ be integers, and write
$\kap:=\lcl(l-1)/(n-2)\rcl-1$. Suppose that $A_1\longc A_{2\kap+1}\seq[0,l]$
are integer sets, having at least $n$ elements each, and such that none of
them are contained in an arithmetic progression with difference greater than
$1$. Then the sumset $A_1\longp A_{2\kap+1}$ contains a block of consecutive
integers of length $2(\kap+1)(n-1)-l>l$.
\end{theorem}

Observe, that Theorem \reft{main} guarantees the existence of a block of
consecutive integers of length $l$ in $A_1\longp A_k$ for $k=2\kap+1$, which
is a sharp threshold: for $k=2\kap$ the sumset may fail to contain such a
block, as witnessed, say, by the system of identical sets
  $$ A_1\longe A_{2\kap}=[0,m]\cup[l-m,l] $$
with $m\in[1,l/2)$ integer and $n=2m+2$. Indeed, in this case we have
  $$ A_1\longp A_{2\kap}
                      = \bigcup_{j=0}^{2\kap} \, [j(l-m),j(l-m)+2\kap m], $$
the length of each individual segment being $2\kap m=\kap(n-2)<l$, and the
segments not abutting, provided $\{(l-1)/(n-2)\}>1/2$.

Unfortunately, we were unable to eliminate the assumption that none of the
sets $A_j$ are contained in an arithmetic progression with difference greater
than $1$. This reminiscent of the condition $\gcd(A)=1$ from Theorem
\reft{lev} will be discussed in Section \refs{conclusion}.

The reader may compare the following corollary against Theorem \reft{szvu}.
\newpage
\begin{corollary}\label{c:main}
Suppose that $k,l\ge 1$ and $n\ge 3$ are integers, and
 $A_1\longc A_k\seq[0,l]$ are integer sets, having at least $n$ elements
each, and such that none of them are contained in an arithmetic progression
with difference greater than $1$. If $k\ge 2\lcl(l-1)/(n-2)\rcl$, then the
sumset $A_1\longp A_k$ contains a block of consecutive integers of length
$k(n-1)$.
\end{corollary}

\begin{proof}
Reducing the value of $l$ and renumbering the sets, if necessary, we can
assume that $0,l\in A_k$. Write $\kap:=\lcl(l-1)/(n-2)\rcl-1$. By Theorem
\reft{main}, the sumset $A_1\longp A_{2\kap+1}$ contains a block of
consecutive integers of length at least $2(\kap+1)(n-1)-l>l$. Adding one by
one the sets $A_{2\kap+2}\longc A_{k-1}$ to this sumset, we increase the
length of the block by at least $n-1$ each time, and adding $A_k$ at the last
step we increase the length by $l$. Consequently, $A_1\longp A_k$ contains a
block of length at least
  $$  (2(\kap+1)(n-1)-l) + (k-2\kap-2)(n-1) + l = k(n-1). $$
\end{proof}

% Consider inserting here a sentence saying that $2\lfl...\rfl+2$ is a sharp
% threshold

\section{Proof of Theorem \reft{main}}\label{s:proofs}

For a finite, non-empty integer set $A$ let $\ell(A)$ denote the difference
of the largest and the smallest elements of $A$.

Our approach is fairly close to that employed in \refb{la}, with the
following result in its heart.
\begin{theorem}[Lev \protect{\cite[Theorem~1]{b:lb}}]\label{t:sumsetgrowth}
Let $k\ge 2$ be an integer, and suppose that $A_1\longc A_k$ are finite,
non-empty integer sets. If $\ell(A_j)\le\ell(A_k)$ for $j=1\longc k-1$ and
$A_k$ is not contained in an arithmetic progression with difference greater
than $1$, then
  $$ |A_1\longp A_k| \ge |A_1\longp A_{k-1}|
                            + \min \{ \ell(A_k),\, n_1\longp n_k-k+1 \}, $$
where
  $$ n_j = \begin{cases}
             |A_j| &\text{ if }\ \ell(A_j)<\ell(A_k) \\
             |A_j|-1 &\text{ if }\ \ell(A_j)=\ell(A_k)
           \end{cases}; \quad j=1\longc k. $$
\end{theorem}

Up to some subtlety which we suppress for the moment, the strategy pursued
below is to apply Theorem \reft{sumsetgrowth} to show that if $k$ is large
enough (which in practice means $k\ge 2\kap+O(1)$), then the densities of the
sumsets $A_1\longp A_{\lfl k/2\rfl}$ and $A_{\lfl k/2\rfl+1}\longp A_k$
exceed $1/2$, and then use the box principle to conclude that the sumset of
the two, which is $A_1\longp A_k$, contains a long block of consecutive
integers. We start with the second, technically simpler, component.

\begin{lemma}\label{l:box}
Let $L_1$ and $L_2$ be positive integers and suppose that $S_1\seq[0,L_1]$
and $S_2\seq[0,L_2]$ are integer sets. If $\max\{L_1,L_2\}\le |S_1|+|S_2|-2$,
then
  $$ [L_1+L_2-(|S_1|+|S_2|-2),|S_1|+|S_2|-2] \seq S_1+S_2. $$
\end{lemma}

\begin{proof}
Given an integer $g\in[0,L_1+L_2]$, the number of representations $g=s_1+s_2$
with arbitrary $s_1\in[0,L_1]$ and $s_2\in[0,L_2]$ is
\begin{align*}
  |(g-[0,L_1]) \cap [0,L_2]|
    &= |[g-L_1,g] \cap [0,L_2]| \\
    &= \min\{g,L_2\} - \max\{g-L_1,0\} + 1 \\
    &= \min\{g,L_2\} + \min\{g,L_1\} - g + 1.
\end{align*}
In order for $g\in S_1+S_2$ to hold, it suffices that this number of
representations exceeds the total number of ``gaps'' in $S_1$ and $S_2$; that
is,
  $$ \min\{g,L_2\} + \min\{g,L_1\} - g + 1 > L_1+L_2+2-|S_1|-|S_2|. $$
This, however, follows immediately for every $g$ in the interval
$[L_1+L_2-(|S_1|+|S_2|-2),|S_1|+|S_2|-2]$, by considering the location of $g$
relative to $L_1$ and $L_2$.
\end{proof}

\begin{comment}
\begin{proof}
Assume for definiteness that $L_1\le L_2$, and let $g$ be an integer.

If $L_1+L_2-(|S_1|+|S_2|-2)\le g\le L_1$, then
\begin{multline*}
 |(g-S_1)\cap[0,g]| + |S_2\cap[0,g]| \\
                        \ge (g+1-(L_1+1-|S_1|)) + (g+1-(L_2+1-|S_2|)) > g+1,
\end{multline*}
whence $g-S_1$ and $S_2$ have non-empty intersection and therefore
 $g\in S_1+S_2$.

If $L_1\le g\le L_2$, then both $g-S_1$ and $S_2$ are subsets of $[0,L_2]$
and
  $$ |g-S_1|+|S_2| = |S_1|+|S_2| > L_2+1, $$
whence $g\in S_1+S_2$ as above.

Finally, if $L_2\le g\le |S_1|+|S_2|-2$, then
  $$ |(g-S_1)\cap[0,L_2]| + |S_2| \ge (|S_1|-(g-L_2)) + |S_2| > L_2+1, $$
implying $g\in S_1+S_2$ in this case, too.
\end{proof}
\end{comment}

We now turn to the more technical part of the argument, consisting in
inductive application of Theorem \reft{sumsetgrowth}.
\begin{proposition}\label{p:ind}
Let $k,l\ge 1$, and $n\ge 3$ be integers, and suppose that $A_1\longc A_k$
are integer sets with $|A_i|\ge n$ and $\ell(A_i)\le l$ for $i=1\longc k$,
such that none of these sets are contained in an arithmetic progression with
difference greater than $1$. Write $S=A_1\longp A_k$.
\begin{itemize}
\item[(i)]  If $k\ge (l-1)/(n-2)-1$, then
  $|S|\ge\frac12\,(\ell(S)+(k+1)(n-1)-l+2)$;
\item[(ii)] if $k\ge (l-1)/(n-2)$, then $|S|\ge\frac12\,(\ell(S)+k(n-1)+2)$.
\end{itemize}
\end{proposition}

Observe, that under assumption (i) we have $(k+1)(n-1)-l+2>0$, so that
$|S|>\ell(S)/2$ in this case. Similarly, under assumption (ii) we have
$k(n-1)+2>l$, and hence in this case the stronger estimate
$|S|>(\ell(S)+l)/2$ holds.

\begin{proof}[Proof of Proposition \refp{ind}]
Write $l_j=\ell(A_j)$. Without loss of generality, we can assume that
$l_1\longle l_k$. By Theorem \reft{sumsetgrowth} and in view of
$\ell(S)=l_1\longp l_k$, we have
\begin{align}
  |S|-\frac12\,\ell(S)
     &\ge \sum_{j=1}^k \big( \min \{ l_j-1,j(n-2) \}  + 1 \big) + 1
                                   - \frac12\,(l_1\longp l_k) \notag \\
     &=   \sum_{j=1}^k \min \lfr \frac{l_j-1}2, j(n-2)-\frac{l_j-1}2 \rfr
           + \frac k2\,+1. \label{e:sum}
\end{align}

Starting with assertion (ii), assume that $k\ge(l-1)/(n-2)$. Color all
integers $j\in[1,k]$ red or blue, according to whether $l_j-1\ge j(n-2)$ or
$l_j-1<j(n-2)$. Notice that the integer $1$ is then colored red, hence the
interval $[1,k]$ can be partitioned into a union of adjacent subintervals
$J_1\longu J_K$ so that each subinterval consists of a block of consecutive
red integers followed by a block of consecutive blue integers, with all
blocks non-empty --- except that the rightmost subinterval $J_K$ may consist
of a ``red block'' only. Accordingly, we write the sum over $j\in[1,k]$ in
the right-hand side of \refe{sum} as $\sig_1\longp \sig_K$, where for each
$\nu\in[1,K]$ by $\sig_\nu$ we denote the sum over all $j\in J_\nu$.

Fixing $\nu\in[1,K]$, write $J_\nu=[s,t]$ and define $q$ to be the largest
red-colored number of the interval $[s,t-1]$. This \emph{does not} define $q$
properly if $\nu=K$ and $s=t=k$; postponing the treatment of this exceptional
case, suppose for the moment that $q$ is well defined. Thus either $q+1$ is
blue, whence $l_{q+1}-1<(q+1)(n-2)$, or $\nu=K,\,t=k$, and $q=k-1$, whence
  $$ l_{q+1}-1=l_k-1\le l-1\le k(n-2)=(q+1)(n-2). $$
Observe, that
\begin{equation}\label{e:lq+1}
  l_{q+1}-1 \le (q+1)(n-2)
\end{equation}
holds in either case, and it follows that
\begin{align*}
  \sig_\nu &=   \sum_{j=s}^q \lpr j(n-2)-\frac{l_j-1}2\rpr
                                          + \sum_{j=q+1}^t \frac{l_j-1}2 \\
           &\ge \frac{n-2}2 \, (q^2+q-s^2+s) - \frac{l_q-1}2 \, (q+1-s)
                                                + \frac{l_q-1}2 \, (t-q) \\
           &=   \frac{n-2}2 \, (q^2+q-s^2+s) + \frac{l_q-1}2 \, (s+t-1-2q).
\end{align*}
We now distinguish two cases: $q\le(s+t-1)/2$ and $q\ge(s+t-1)/2$. In the
former case we have
\begin{align*}
  \sig_\nu &\ge \frac{n-2}2\,(q^2+q-s^2+s+q(s+t-1-2q)) \\
           &=   \frac{n-2}2\,(q(s+t-q)-s^2+s) \\
           &=   \frac{n-2}2\,(st+(q-s)(t-q)-s^2+s) \\
           &\ge \frac{n-2}2\,(st-s^2+s) \\
           &=   \frac{n-2}2\,s|J_\nu|.
\end{align*}
In the latter case, taking into account that
 $l_q-1\le l_{q+1}-1\le (q+1)(n-2)$ by \refe{lq+1}, we obtain the same
estimate:
\begin{align*}
  \sig_\nu &\ge \frac{n-2}2\,(q^2+q-s^2+s+(q+1)(s+t-1-2q)) \\
           &=   \frac{n-2}2\,(q(s+t-2-q)-s^2+2s-1+t) \\
           &\ge \frac{n-2}2\,((t-1)(s-1)-s^2+2s-1+t) \\
           &=   \frac{n-2}2\,(st-s^2+s) \\
           &=   \frac{n-2}2\,s|J_\nu|.
\end{align*}
Addressing finally the situation where $\nu=K$ and $s=t=k$, we observe that
in this case
\begin{align*}
  \sig_\nu &=   k(n-2) - \frac{l_k-1}2 \\
           &\ge k(n-2) - \frac{l-1}2 \\
           &\ge \frac{k(n-2)}2 \\
           &=   \frac{n-2}2\,s|J_\nu|,
\end{align*}
as above. Thus,
  $$ \sig_1\longp\sig_K \ge \frac{n-2}{2}\ \big( |J_1|\longp|J_K| \big)
                                                  = k \, \frac{n-2}{2}, $$
and substituting this into \refe{sum} we obtain assertion (ii).

To prove assertion (i), instead of $k\ge(l-1)/(n-2)$ assume now the weaker
bound $k\ge(l-1)/(n-2)-1$. Set $l_{k+1}=l_k$, so that $l_{k+1}-1\le
l-1\le(k+1)(n-2)$. From \refe{sum} we get
  $$ |S|-\frac12\,\ell(S)
       \ge \sum_{j=1}^{k+1} \min \lfr \frac{l_j-1}2, j(n-2)-\frac{l_j-1}2 \rfr
                                       - \frac{l_{k+1}-1}2 + \frac k2\,+1. $$
Since $k+1\ge (l-1)/(n-2)$, the sum over $j$ can be estimated as above, and
it is at least $(k+1)(n-2)/2$. Assertion (i) now follows in view of
\begin{multline*}
  \frac{(k+1)(n-2)}2-\frac{l_{k+1}-1}2+\frac k2+1 \\
         \ge \frac12\,\big( (k+1)(n-2) - (l-1) + k + 2\big)
                  = \frac12\,\big( (k+1)(n-1) - l + 2 \big).
\end{multline*}
\end{proof}

Proposition \refp{ind} took us most of the way to the proof of Theorem
\reft{main}.
\begin{proof}[Proof of Theorem \reft{main}]
Assume that the sets $A_1\longc A_{2\kap+1}$ are so numbered that, letting
$l_j=\ell(A_j)$ for $j\in[1,2\kap+1]$, we have
 $l_1\ge l_2\longge l_{2\kap+1}$. We are going to partition our sets into two
groups to apply Lemma \refl{box} to the sumsets $S_1$ and $S_2$ of these
groups, and this is to be done rather carefully as effective application of
the lemma requires that $S_1$ and $S_2$ be of nearly equal length.

Accordingly, we let $S_1:=A_1+A_3\longp A_{2\kap-1}$ (including all sets
$A_j$ with odd indices $j<2\kap$) and
 $S_2:=(A_2+A_4\longp A_{2\kap})+A_{2\kap+1}$ (all sets $A_j$ with even
indices $j\le 2\kap$ and $A_{2\kap+1}$). By Proposition \refp{ind} we have
\begin{align}
  |S_1| &\ge \frac12\,(\ell(S_1)+(\kap+1)(n-1)-l+2) \label{e:sls1} \\
\intertext{and}
  |S_2| &\ge \frac12\,(\ell(S_2)+(\kap+1)(n-1)+2). \label{e:sls2}
\end{align}
Furthermore, from
  $$ \ell(S_1)-\ell(S_2)=(l_1-l_2)\longp(l_{2\kap-1}-l_{2\kap})
                                                  -l_{2\kap+1}\in[-l,l], $$
using \refe{sls1} and \refe{sls2} we get
\begin{multline*}
  \max \{ \ell(S_1),\ell(S_2) \} \le \frac12\,(\ell(S_1)+\ell(S_2)+l) \\
    \le |S_1|+|S_2| +l - (\kap+1)(n-1) -2 \le |S_1|+|S_2| - 2.
\end{multline*}
Applying Lemma \refl{box} and using again \refe{sls1} and \refe{sls2}, we
conclude that $A_1\longp A_k=S_1+S_2$ contains a block of consecutive
integers of length at least
  $$ 2(|S_1|+|S_2|-2) - (\ell(S_1)+\ell(S_2))
       \ge 2(\kap+1)(n-1) -l , $$
as required.
\end{proof}

\section{Concluding remarks and open problems}\label{s:conclusion}

The major challenge arising in connection with the main results of this paper
(which are Theorem \reft{main} and Corollary \refc{main}) is to get rid of
the assumption that none of the sets involved are contained in an arithmetic
progression with difference greater than $1$. One can expect that a vital
ingredient of such an improvement would be a suitable generalization of
Theorem \reft{sumsetgrowth}. Indeed, we were able to generalize Theorem
\reft{sumsetgrowth} in what seems to be the right direction.
\begin{primetheorem}\label{t:classes}
Let $k\ge 2$ be an integer, and let $A_1\longc A_k$ be finite, non-empty
integer sets. Suppose that $l\in A_k-A_k$ is a positive integer, and define
$d$ to be the largest integer such that $A_k$ is contained in an arithmetic
progression with difference $d$. Then
  $$ |A_1\longp A_k| \ge |A_1\longp A_{k-1}|
                              + \min \{ hl/d,\, n_1\longp n_k-k+1 \}, $$
where $h$ is the number of residue classes modulo $d$, represented in
$A_1\longp A_{k-1}$, and $n_j\ (j=1\longc k)$ is the number of residue
classes modulo $l$, represented in $A_j$.
\end{primetheorem}

Clearly, Theorem \reft{classes} implies Theorem \reft{sumsetgrowth} and,
furthermore, shows that the assumption of Theorem \reft{main} that none of
the sets $A_j$ are contained in an arithmetic progression with difference
greater than $1$ can be slightly relaxed. Specifically, it suffices to
request that the $A_j$ can be so ordered that $\ell(A_j)$ increase, and for
every $k\in[1,2\kap]$ the sumset $A_1\longp A_k$ represents all residue
classes modulo $\gcd(A_{k+1}-A_{k+1})$. It seems, however, that Theorem
\reft{classes} by itself fails short to extend Theorem \reft{main} the
desired way, dropping the modular restriction altogether, and in the absence
of an application we do not present here the proof of the former theorem.

Another interesting direction is to refine Theorem \reft{main} as to the
length of the block, contained in the sumset $A_1\longp A_{2\kap+1}$. While
we observed that $2\kap+1$ is the smallest number of summands which ensures a
block of length $l$, it is quite possible that the existence of a longer
block can be guaranteed. In this connection we mention the following
conjecture from \refb{la}, referring to the equal summands situation.
\begin{conjecture}\label{j:lint}
Let $k,l\ge 1$ and $n\ge 3$ be integers, and write
$\kap:=\lfl(l-1)/(n-2)\rfl$. Suppose that $A\seq[0,l]$ is an integer set with
$0,l\in A,\,\gcd(A)=1$, and $|A|=n$. If $k\ge 2\kap+1$, then $kA$ contains a
block of consecutive integers of length $(k-\kap)l+k((\kap+1)(n-2)+2-l)$.
\end{conjecture}
In fact, Conjecture \refj{lint} is established in \refb{la} in the case where
$k\ge 3\kap$, but the case $2\kap+1\le k<3\kap$, to our knowledge, remains
open.

\medskip

\end{document}